\documentclass[12pt]{article}
\usepackage[utf8]{inputenc}
\usepackage{amsmath, amsthm, amsfonts, 
mathrsfs,enumerate,bm}
\usepackage{color}
\usepackage{ascmac}
\newtheorem{theorem}{Theorem}[section]
\newtheorem{lemma}[theorem]{Lemma}
\newtheorem{proposition}[theorem]{Proposition}
\newtheorem{cor}[theorem]{Corollary}

\newtheorem{ex}[theorem]{Example}
\newtheorem{ass}[theorem]{Assumption}

\allowdisplaybreaks[3]

\makeatletter
 
 \@addtoreset{equation}{section}
\makeatother

\title{The SIML method without microstructure noise}
\author{
Jir\^o Akahori\footnote{
Department of Mathematical Science, 
Ritsumeikan University, 
1-1-1 Nojihigashi, Kusatsu, Shiga, 525-8577, Japan 
(e-mail: {\tt akahori@se.ritsumei.ac.jp})},
Ryuya Namba\footnote{
Department of Mathematical Science,
Kyoto Sangyo University, 
Motoyama, Kamigamo, 
Kita-ku, Kyoto, 603-8555 Japan 
(e-mail: {\tt rnamba@cc.kyoto-su.ac.jp})}, 
and Atsuhito Watanabe\footnote{
Kusatsu 525-8529, Japan
Garduate School of Science and Engineering, Ritsumeikan University, 1-1-1, Noji-Higashi, Kusatsu, Shiga, 525-8577, Japan 
(e-mail: {\tt atsu.watanabe0507@gmail.com})} \footnote{Corresponding author}}
\date{}

\begin{document}

\maketitle
\begin{abstract}
The SIML (abbreviation of Separating Information Maximal Likelihood) method, 
has been introduced by N. Kunitomo and S. Sato and their collaborators to estimate the integrated volatility of high-frequency data that is assumed to be an It\^o process
but with so-called microstructure noise. 
The SIML estimator turned out to share many properties
with the estimator introduced by P. Malliavin and M.E. Mancino. 
The present paper 
establishes the consistency and the asymptotic normality 
under a general sampling scheme but without microstructure noise. 
Specifically, a fast convergence shown for Malliavin--Mancino estimator by E. Clement and A. Gloter
is also established for the SIML estimator. 

\vspace{2mm}

\vspace{2mm}
\end{abstract}

\section{Introduction}

\subsection{The Problem}\label{Problem}
Throughout the present paper, we consider a complete probability space 
$ (\Omega, \mathcal{F}, \mathbf{P}) $,
which supports a $d$-dimensional Wiener process $ \mathbf{W} \equiv (W^1, W^2, \dots, W^d) $ on the time interval $[0,1]$.
We denote by $ \mathcal{F}_t,\, t \in [0, 1], $
the complete $ \sigma $-algebra generated by 
$ \{\mathbf{W}_s : 0 \leq s \leq t \} $ and by $ L^2_a [0,1] $
the space of $ \{\mathcal{F}_t \} $-adapted processes $ \theta $ with 
$ \mathbf{E} [\int_0^1 |\theta(s)|^2 \mathrm{d}s ] < +\infty $. 

Let $J \in \mathbb{N}$. Consider an It\^o process 
\begin{equation}\label{Ito}
\begin{split}
    X^j_t = X^j_0 
    + \int_0^t b^j (s) \,\mathrm{d}s 
    + \sum_{r=1}^d \int_0^t \sigma^j_r (s) \, \mathrm{d} W^r_s,
\end{split}
\end{equation}
for $ j=1, 2, \dots, J $ and $ t \in [0,1] $, 
where $ b^j, \sigma^j_r \in L^2_a [0,1] $ 
for all $ j =1, 2, \dots, J$ 
and $ r =1, 2, \dots, d $. 

We take 
the observations 
for the $ j $-th component 
of the process at time $ 0 = t^j_0 <t^j_1 < \cdots< t^j_{n_j} = 1 $ for $ j = 1, 2, \dots, J$. Here we conventionally assume that we observe the initial price and the final price but the assumption can be relaxed. 
We are interested in  
constructing an estimator 
$ (V^{j,j'})_{j, j'=1, 2, \dots, J} $
of
integrated volatility matrix defined by
\begin{equation*}
    \int_0^t \Sigma^{j,j'} (s) \,\mathrm{d}s
    :=\sum_{r=1}^d \int_0^t \sigma_r^j (s) \sigma_r^{j'} (s)\, \mathrm{d}s,
    \qquad t \in [0, 1],
\end{equation*}
out of the observations, 
which is consistent 
in the sense that each
$ V^{j,j'} $ converges to 
$ \int_0^t \Sigma^{j,j'} (s) \,{\rm d}s $
in probability as $ n := \min_{1\leq j \leq J} n_j \to \infty $, under the condition that 
\begin{equation}\label{rho}
    \rho_n := \max_{j,k} |t^j_k - t^j_{k-1}| \to 0
\end{equation}
as $ n \to \infty $.

\subsection{SIML method}\label{SIMLintro}

Let us briefly review 
the {\it separating information maximum likelihood} (SIML for short) estimator, introduced 
by N. Kunitomo together with his collaborator S. Sato 
in a series of papers \cite{KS08a, KS08b, KS10, KS2, KS13} 
where the observations are assumed to be with {\it microstructure noise}. 
Namely, the observations are 
\begin{equation}
    Y^j (t^j_k) \equiv X^j (t^j_k)+ v^j_k \label{microstructure noise}
\end{equation}
for $ k = 0, 1, \dots, n_j $ and $ j = 1, 2, \dots, J$, 
where $ \{v^j_k\}_{j, k} $
is a family of zero-mean $i.i.d.$ random variables with finite fourth moment, 
which are independent 
of the Wiener process $ \mathbf{W} $. 

Let the observations 
be equally spaced, that is, $ t^j_k \equiv k/n $. 
The estimator of the SIML method  
is given by 
\begin{equation}\label{SIMLestimator}
    \begin{split}
        V_{n,m_n}^{j,j'}
        := \frac{n}{m_n} \sum_{l=1}^{m_n} \left( \sum_{k=1}^{n^j}
        p^{n_j}_{k,l} \Delta Y^j_k \right)
        \left(\sum_{k'=1}^{n^{j'}} p^{n_{j'}}_{k',l} \Delta Y^{j'}_{k'} \right),  
    \end{split}
\end{equation}
where $ m_n (\ll n) $ is an integer,
\begin{equation*}
    p^n_{k,l} = \sqrt{\frac{2}{n+ \frac{1}{2}}}
    \cos \left( \left(l - \frac{1}{2}\right)
    \pi \left( \frac{k-\frac{1}{2}}{n+\frac{1}{2}}
    \right) \right)
\end{equation*}
for $ k,l=1, 2, \dots, n $, $ n \in \mathbf{N} $, 
and we understand 
$ \Delta $ to be the difference operator given by
$ (\Delta a)_k = a_k -a_{k-1} $ for a sequence $\{a_k\}_k$. 
We then write 
\begin{equation*}
\Delta Y^j_k = Y^j_{t^j_{k}} - Y^j_{t^j_{k-1}}. 
\end{equation*}

They have proved the following two properties.
\begin{enumerate}[{\bf (i)}]
    \item ({\it the consistency}):
    the convergence in probability of $ V^{j,j'}_{n, m_n} $ to $ \int_0^1 \Sigma^{j,j'} (s) \,\mathrm{d}s $ as $ n \to \infty $
    is attained, provided 
    that $ m_n = o (n^{1/2})$, and
    \item ({\it the asymptotic normality of the error}): the stable convergence of 
    \begin{equation*}
    \begin{split}
       & \sqrt{m_n}\left(V^{j,j'}_{n, m_n}-\int_0^1 \Sigma^{j,j'} (s) \,\mathrm{d}s \right) \\
        & \to N \left(0, \int_0^1 \left(\Sigma^{j,j}(s) \Sigma^{j',j'}(s) + (\Sigma^{j,j'} (s))^2 \right) \mathrm{d}s
        \right)
        \end{split}
    \end{equation*}
    holds true as $ n \to \infty$ 
    if $ m_n = o (n^{2/5})$,
\end{enumerate}
under some mild conditions on $ b $
and $ \Sigma $. See \cite{KSK} for more details. 
In the book \cite{KSK}, more properties of the SIML estimator are proven. 
Here we just pick up some of them. 

\subsection{SIML as a variant of Malliavin--Mancino method} 

The Malliavin--Mancino's Fourier (MMF for short) method, 
introduced in \cite{MMFS} and \cite{MMAS}, 
is an estimation method 
for the spot volatility $ \Sigma^{j, j'} (s) $
appeared in Section \ref{Problem}, 
by constructing an estimator of the 
Fourier series of $ \Sigma^{j, j'} $. 
The series consists of estimators 
of Fourier coefficients
given by 
\begin{equation}\label{Fourier-estimator}
    \begin{split}
\widehat{\Sigma^{j,j'}_{n,m_n}}(q)  
:=\frac{1}{m_n} \sum_{l=1}^{m_n} \left( \sum_{k=1}^{n^j}e^{2\pi\sqrt{-1} (l+q) t^j_{k-1} }
        \Delta Y^j_k \right)
        \left(\sum_{k'=1}^{n^{j'}}e^{-2\pi\sqrt{-1} l t_{k'-1}^{j'} }  \Delta Y^{j'}_{k'} \right) 
    \end{split}
\end{equation}
for $ q \in \mathbf{Z} $. 
As we see, 
$ \widehat{\Sigma^{j,j'}_{n,m_n}}(0)$
is quite similar to 
the SIML estimator \eqref{SIMLestimator}. 

The main concern of the SIML estimator 
is to eliminate the microstructure noise, 
and it was derived from a heuristic observation that it might maximize 
a virtual likelihood function (see \cite[Chapter 3, Section 2]{KSK}). 
On the other hand, 
the MMF method aims at the estimation of 
spot volatilities, though the cut-off effects 
have been recognized well among the Italian school, especially by
 M. Mancino and S. Sanfelici (see \cite{MS1} and \cite{MS2})\footnote{
It has been pointed out  
that the ``bias"
$
    \mathbf{E} [ V^{j,j'} - \widehat{\Sigma^{j,j'}_{n,m_n}}(0)]
$
converges to zero and 
 the mean square error
$
    \mathbf{E} [ (V^{j,j'} - \widehat{\Sigma^{j,j'}_{n,m_n}}(q) )^2
$
does not diverge
when $ m_n = o (n) $ as $ n \to \infty $,
which is not the case with the realized volatility.
}. 
Nonetheless, the two methods reached to a similar solution, {\em independently}. 
This is really striking and worth further investigations.

The task of the present paper is 
to establish limit theorems for the SIML estimator, 
given below as \eqref{SIMLestimator2} 
with a more general sampling scheme than \eqref{KSSS},
under the {\it no-microstructure noise circumstance}.
We mostly employ the techniques from \cite{ClGl}.
Some of them are directly applicable to our framework, 
but some are not. 
The main difficulty comes from 
the nature of the kernel \eqref{new Dirichlet1}.
Unlike the Dirichlet kernel, its integral over $ [0,1] $
is not unit $ \times 1/m $, which causes some serious troubles. 
Among the contributions of the present paper, 
establishing the fast convergence corresponding to the one studied in \cite{ClGl} 
as well as the limit theorems under the general sampling scheme is to be the most important one. 
The study of the limit theorems under the general sampling scheme with the cases with microstructure noise is postponed to a forthcoming paper.

\subsection{Organization of the rest of the present paper}
The rest of the present paper is divided into two parts. 
The former part, Section \ref{sec:GFE}, studies 
the consistency of the estimator. The latter part, Section \ref{sec:LTS}, investigates the asymptotic normality of the estimator.
Both sections are structured to be {\em pedagogical}.
Explaining the intuitions behind the setting and the assumptions for the main theorems, the essence of the proof is given in advance of the statement. 
The proofs are given concisely in the last subsection.

\section{Consistency of the SIML estimator in the absence of microstructure noise}\label{sec:GFE}

\subsection{Setting}
To state our results and to give proofs for them in a neat way,
we restate the setting with some new notations. 
First, for a given 
observation time grid
$ \Pi := \{ (t_k^j)_{k= 0,1,\cdots,n_j}: j=1, 2, \dots, J \} $, 
we define 
\begin{equation*}
\begin{split}
    {\Pi}^* &:= \{\varphi = (\varphi_1 (s), \cdots, \varphi_J(s)) : [0,1] \to [0,1]^J \mid \text{{\bf (A1)} and {\bf (A2)}} \}.
\end{split}
\end{equation*}
where we put

\begin{itemize}
\item[{\bf (A1):}]
    The image $\varphi_j ([t^j_{k-1}, t^j_{k}) )$
    is one point in $[t^j_{k-1}, t_k^j ]$ for $k=1, 2, \dots, n_j$ and $j=1, 2, \dots, J $, 
\item[{\bf (A2):}]
    It holds that $ \varphi^j ([t^j_{k-1}, t^j_{k}) ) \ne \varphi^j ([t^j_{k}, t^j_{k+1}) )$ for $k=1, 2, \dots, n_j$ and $j=1, 2, \dots, J $.
\end{itemize}
By using a function in $ \Pi^* $, we can rewrite the Riemann sums in \eqref{SIMLestimator} as stochastic integrals for which It\^o's formula is applicable. 

As remarked in the introduction, we will be working on the 
situations where $ v^j_K \equiv 0 $ henceforth.
Thus, the SIML estimator \eqref{SIMLestimator} can now be redefined as
    \begin{align}
        &V_{n,m_n}^{j,j'} \nonumber\\
        &:= \frac{2n}{n+\frac{1}{2}}\frac{1}{m_n} \sum_{l=1}^{m_n} \left(\int_0^1 
        \cos \left(l-\frac{1}{2} \right) \pi \varphi^j (s) \,\mathrm{d}X^j_s
        \right) \left(\int_0^1 
        \cos \left(l-\frac{1}{2} \right) \pi \varphi^{j'} (s) \, \mathrm{d}X^{j'}_s
        \right),
        \label{SIMLestimator2}
    \end{align}
where $ \varphi \in ((k/n)_{k=1}^n, \cdots,(k/n)_{k=1}^n)^* $ is defined by 
\begin{equation}\label{KSSS}
\varphi^j \left( \left[\frac{k-1}{n}, \frac{k}{n}\right) \right) = \frac{2k-1}{2n+1}
= \frac{1}{n} \left( k-1 + \frac{n-k+1}{2n+1}\right) \in \left[\frac{k-1}{n}, \frac{k}{n}\right)
\end{equation}
for $ k=1, 2, \dots, n $ and $ j=1, 2, \dots, J $. 

In the sequel, we rather work on general sampling scheme, that is, 
general $ \Pi $ and $ \varphi \in \Pi^*$, 
under the condition of \eqref{rho}. 
In doing so, 
the equation \eqref{SIMLestimator2} is the definition of the estimator $ V^{j,j'}_{n,m_n} $, leaving \eqref{SIMLestimator}
as a special case. 

We also introduce a symmetric kernel $ \mathcal{D}^{j,j'}_m : [0,1]\times[0,1] \to \mathbf{R} $ associated with $ \varphi \in \Pi^* $ by
\begin{align}\label{new Dirichlet1} 
  \begin{aligned}
        \mathcal{D}_{m}^{j, j'}(u, s)
       &:=\frac{1}{2m}\frac{\sin{m\pi\left(\varphi^j(u)+\varphi^{j'}(s)\right)}}{\sin{\pi\left(\varphi^j(u)+\varphi^{j'}(s)\right)/2}} + \frac{1}{2m}\frac{\sin{m\pi\left(\varphi^j(u)-\varphi^{j'}(s)\right)}}{\sin{\pi\left(\varphi^j(u)-\varphi^{j'}(s)\right)/2}} \\
  \end{aligned}
    \end{align} 
for $ u,s \in [0,1] $. 
Then, by applying It\^o's formula to the 
products of the stochastic integrals in \eqref{SIMLestimator2}, we have 
\begin{align}\label{ItoF}
    \begin{aligned}
 \frac{n+\frac{1}{2}}{n}{V}_{n,m_n}^{j,j'} &= 
    \int_0^1 \mathcal{D}_{m_n}^{j, j'}(s, s) \, \Sigma^{j,j'} (s) \, 
    \mathrm{d}s
    + \left( \int_0^1 \int_0^s +
   \int_0^1 \int_0^u \right)
     \mathcal{D}_{m_n}^{j, j'}(u, s) \,  {\rm d}X^j_u {\rm d}X^{j'}_s  
    \end{aligned}
\end{align}
since 
\begin{align}\label{key equality}
        &\frac{2}{m}\sum_{l=1}^{m}\cos\left(l-\frac{1}{2}\right)\pi u \cos\left(l-\frac{1}{2}\right)\pi s \nonumber\\
        &=
       \frac{1}{2m}\frac{\sin{m\pi\left(u+s)\right)}}{\sin{\pi\left(u+s\right)/2}} + \frac{1}{2m}\frac{\sin{m\pi\left(u-s\right)}}{\sin{\pi\left(u-s\right)/2}}
        \qquad u, s \in [0, 1].
\end{align}

\subsection{Discussions for possible sampling schemes}
In this section, we will discuss 
how the sampling scheme $ \Pi $ and $ \varphi \in \Pi^* $ should be. 
As we will see, 
we necessarily have 
\begin{align}\label{sampligcond}
    \int_0^1 \mathcal{D}_{m_n}^{j, j'}(s, s) g (s) \,\mathrm{d}s
    \to \int_0^1 g(s) \, \mathrm{d}s \quad\text{as $ n \to \infty $ for any $ g \in C [0,1]$.}
\end{align}
to obtain $ V^{j,j'}_{n, m_n} \to \int_0^1 \sigma^{j,j'} (s) \,\mathrm{d}s $ in probability. 

First, 
we consider the cases where
\begin{align}\label{MC}
   \text{ $ m_n \to \infty$ as $ n \to \infty$}
\end{align}
and 
\begin{align}\label{MMC}
   \text{$ \rho_n m_n \to 0 $ as $ n \to \infty$.}
\end{align}
\begin{lemma}\label{l21}
Under the conditions \eqref{rho},
\eqref{MC} and 
\eqref{MMC}, we have \eqref{sampligcond}. 
\end{lemma}
\begin{proof}
Put 
\begin{align}
        \mathcal{D}_m (u,s) 
        &:= \frac{2}{m}\sum_{l=1}^{m}\cos\left(l-\frac{1}{2}\right)\pi u \cos\left(l-\frac{1}{2}\right)\pi s \label{D-exp1}\\
       & = \frac{1}{m} \sum_{l=1}^{m}
      \left( \cos \left(l-\frac{1}{2}\right)\pi (u+s) 
       + \cos \left(l-\frac{1}{2}\right)\pi (u-s) \right)
        \qquad u, s \in [0, 1] \nonumber
\end{align}
Then, on one hand, we have 
\begin{align*}
    \mathcal{D}_{m}^{j, j'}(u, s) = \mathcal{D}_m (\varphi^j(u),
    \varphi^{j'}(s) ), 
\end{align*}
and 
\begin{align*}
    \begin{aligned}
        |\mathcal{D}_{m}^{j, j'}(u, s) -\mathcal{D}_m (u,s) )| \leq  2m|\varphi^j (u)-u| + 2m|\varphi^{j'} (s) - s|, 
        \qquad u, s \in [0, 1],
    \end{aligned}
\end{align*}
since it holds in general that
\begin{align*}
|\cos c x - \cos c y|\leq c|x-y|  
\end{align*}
for a constant $ c > 0 $.
Therefore, under the assumption \eqref{rho}, 
\begin{align*}
    \begin{aligned}
        & \int_0^1 (\mathcal{D}_{m_n}^{j, j'}(s, s) -\mathcal{D}_{m_n} (s,s) )g (s) \,\mathrm{d}s \leq 
        4 \rho_n m_n \Vert g \Vert_{L^2}
        \to 0 
    \end{aligned}
\end{align*}
as $ n \to \infty $. 
On the other hand, since 
\begin{align*}
    \begin{aligned}
       \mathcal{D}_{m}(s, s) 
       &= 1 +  \frac{1}{2m}\frac{\sin (2m\pi s)}{\sin{(\pi s)}}, 
    \end{aligned}
\end{align*}
we have
\begin{align*}
\begin{aligned}
 &  \int_0^1 \mathcal{D}_{m}(s, s) g (s) \,\mathrm{d}s 
    -\int_0^1 g (s) \, \mathrm{d}s = \frac{1}{2m}  \int_0^1 g (s)\frac{\sin (2m\pi s)}{\sin{(\pi s)}} \, \mathrm{d}s
\end{aligned}
\end{align*}
Since
\begin{align*}
    \left| \frac{\sin (2m\pi s)}{\sin{(\pi s)}}\right| \leq 
    \begin{cases}
    \frac{1}{2 \pi s}\,
    & s \in (0, 1/2] 
\\
     \frac{1}{2\pi (1-s)} & s \in [1/2, 1)
    \end{cases},
\end{align*}
by setting $ A_\varepsilon := [0,\varepsilon)  \cup (1-\varepsilon, 1]$, 
we have 
\begin{align}\label{loge}
\begin{aligned}
   & \int_{[0,1]\setminus A_\varepsilon} \left| \frac{\sin (4m\pi s)}{2m \sin{(2\pi s)}}\right| \, \mathrm{d}s 
   \leq 
   \int_{[\varepsilon, 1/2]}\frac{1}{4m\pi s} \,
   \mathrm{d}s +\int_{[1/2, 1-\varepsilon]} \frac{1}{4m\pi (1-s)} \, 
   \mathrm{d}s \\
   & \leq \frac{1}{2m\pi} (\log \varepsilon^{-1} - \log 2)
   \end{aligned}
\end{align}
for arbitrary $ \varepsilon \in (0, 1/4) $.
Using \eqref{loge} and the bound 
\begin{align*}
    \left| \frac{\sin (2m\pi s)}{2m \sin{(\pi s)}}\right| = \left|
    \frac{1}{m} \sum_{l=1}^m \cos (2l-1) \pi s \right| \leq 1 
\end{align*}
on $ A_\varepsilon $, 
we obtain 
\begin{align*}
    \begin{aligned}
         \left| \int_0^1 g (s)\frac{\sin (2m\pi s)}{\sin{(\pi s)}} \, \mathrm{d}s \right|
        \leq \Vert g \Vert_\infty 
      \left(- \frac{\log (2\varepsilon)}{2m\pi}
      + 2 \varepsilon \right).
    \end{aligned}
\end{align*}
In particular, 
by taking $ \varepsilon = m^{-1} $
for $ m > 4 $, 
we see that, for $ \alpha < 1$,
\begin{align}\label{qDir}
\begin{aligned}
 &  m^{\alpha}_n \left(\int_0^1 \mathcal{D}_{m}(s, s) g (s) \,\mathrm{d}s 
    -\int_0^1 g (s) \mathrm{d}s\right) \to 0  \quad
    \text{as $ m_n, n \to \infty $. }
\end{aligned}
\end{align}
Given the above two observations, 
the proof is complete since 
\begin{align*}
    \begin{aligned}
         \int_0^1 \mathcal{D}_{m_n}^{j, j'}(s, s) g (s) \,\mathrm{d}s 
         =  \int_0^1 \mathcal{D}_{m_n}(s, s) g (s) \,\mathrm{d}s 
         + \int_0^1 (\mathcal{D}_{m_n}^{j, j'}(s, s) -\mathcal{D}_{m_n} (s,s) )g (s) \,\mathrm{d}s. 
    \end{aligned}
\end{align*}

\end{proof}

To work on the ``optimal rate" 
(see \cite[A3]{ClGl}, see also \cite[Remark 3.2]{MRS})
\begin{equation}\label{H4}
        0< \liminf_{m_n, n \to \infty} m_n\rho_n \leq \limsup_{m_n, n \to \infty} m_n\rho_n < \infty.  
\end{equation}
we need to assume \eqref{sampligcond} 
instead of proving. This is the strategy taken in \cite{ClGl}. 
Proposition \ref{eq-sp} below 
justifies the strategy. 

For integers $ l$,
we denote by $ [ l ]_n$
its remainder of the division by $ n $,  
that is,
$ [l]_n \equiv l (\mod n) $ with the property $ 0 \leq [l]_n < n $. 

\begin{proposition}\label{eq-sp}
Let $ t^j_k \equiv k/n $.  

\vspace{1mm}
\noindent
{\rm (i)} Let  
$ \varphi^j ([t^j_{k-1}, t^j_k)) = t^j_{k-1} $
for all $ j $ and $ k $
and assume that $ m_n \to \infty $ and $ [m_n]_n/m_n \to 0 $ as $ n \to \infty $. 
Then, the statement \eqref{sampligcond} holds true. 

\vspace{1mm}
\noindent
{\rm (ii)} On the contrary, let $ J \geq 2 $, $ \varphi^j  ([t^j_{k-1}, t^j_k)) = t^j_{k-1} $ while 
$ \varphi^{j'}  ([t^{j'}_{k-1}, t^{j'}_k)) = t^{j'}_{k} $
for some $ 1 \leq j \neq j' \leq J $. Then, when $ m_n = 2n $,
$ \int_0^1 \mathcal{D}_{m_n}^{j, j'}(s, s) \,\mathrm{d}s \equiv 0 $, that is, 
\eqref{sampligcond} fails to be true. 
\end{proposition}

\begin{proof}
(i) First we note that in this case 
\begin{align*}
    \begin{aligned}
      \int_0^1 \mathcal{D}_{m_n}^{j, j'}(s, s) g(s) \,\mathrm{d}s 
        = \int_0^1 g(s) \,\mathrm{d}s  + \frac{1}{m_n } \sum_{l=1}^{m_n} \sum_{j=1}^{n} \cos \left( (2l-1) \pi \frac{j-1}{n}\right) \int_{(j-1)/n}^{j/n} g(s) \,\mathrm{d}s.
    \end{aligned}
\end{align*}
By denoting $ \zeta_{2n} = e^{\sqrt{-1}\pi/n}$,
we see that
\begin{align*}
    \begin{aligned}
       & \sum_{l= cn+1}^{(c+1)n} \cos \left( (2l-1) \pi \frac{j-1}{n} \right)
       =  \frac{1}{2}\sum_{l= cn+1}^{(c+1)n} (\zeta_{2n}^{(2l-1)(j-1)} + \zeta_{2n}^{(2l-1)(2n-j+1)})= 0
     \end{aligned}
\end{align*}
for $ c \in \mathbf{N} $ and $ j \neq 1 $.
Then, 
\begin{align*}
    \begin{aligned}
      &  \left| \frac{1}{m_n }\sum_{l=1}^{m_n} \sum_{j=1}^{n} \cos \left( (2l-1) \pi \frac{j-1}{n}\right) \int_{(j-1)/n}^{j/n} g(s) \,\mathrm{d}s\right| \\
      & = \left|\frac{1}{m_n }\sum_{l=1}^{[m_n]_n} \sum_{j=1}^{n} \cos \left( (2l-1) \pi \frac{j-1}{n}\right) \int_{(j-1)/n}^{j/n} g(s) \,\mathrm{d}s + \frac{m_n- [m_n]_n}{m_n}\int_0^{1/n} g(s)\,\mathrm{d}s \right| \\
      & \leq 
      \frac{1}{m_n } \sum_{j=1}^{n}
      \sum_{l=1}^{[m_n]_n}
      \left|\cos \left( (2l-1) \pi \frac{j-1}{n}\right)\right| \int_{(j-1)/n}^{j/n} |g(s)| \,\mathrm{d}s +  \int_0^{1/n} |g(s)|\,\mathrm{d}s \\
      &\leq \left( \frac{[m_n]_n}{m_n}
      + \frac{1}{n} \right)\Vert g \Vert_2,
    \end{aligned}
\end{align*}
which converges to zero as $ n \to \infty $ by the assumption. 

(ii) In this case, 
\begin{align*}
    \begin{aligned}
     & \int_0^1 \mathcal{D}_{m_n}^{j, j'}(s, s)\,\mathrm{d}s \\
         & = \frac{1}{n m} \sum_{j=1}^n\sum_{l=1}^{2n}
      \left( \cos \left(l-\frac{1}{2}\right)\pi \left(\frac{2j-1}{n} \right) 
       + \cos \left(l-\frac{1}{2}\right)\pi \left(\frac{1}{n}\right) \right) \\
   &= \frac{1}{2n m} \sum_{j=1}^n\sum_{l=1}^{2n}
   (\zeta_{4n}^{(2l-1)(2j-1)} + \zeta_{4n}^{(2l-1)(4n-2j+1)}
   + \zeta_{4n}^{(2l-1)} + \zeta_{4n}^{(2l-1)}) =0.
    \end{aligned}
\end{align*}
\end{proof}

\subsection{Discussions for the residues}
Given the discussions in the previous subsection, the consistency of the estimator $ V^{j,j'}_{n, m_n} $
is now reduced to the convergence 
(to zero) of the residue terms 
\begin{align}\label{Fundecomp}
    \begin{aligned}
         & \left(\int_0^1 \int_0^s
        + \int_0^t\int_0^u \right) \mathcal{D}_{m_n}^{j, j'}(u, s)\,     
        {\rm d}X^j_u {\rm d}X^{j'}_s
        \\
         &= M_{m_n}^{j,j'}(1) 
         + M_{m_n}^{j',j}(1) + I_{m_n}^{1,j,j'}(1) + I_{m_n}^{2,j,j'} (1) + I_{m_n}^{3,j,j'} (1) \\
         &\hspace{1cm}+I_{m_n}^{1,j',j}(1) + I_{m_n}^{2,j',j} (1) + I_{m_n}^{3,j',j} (1) 
\end{aligned}
\end{align}
for $ j, j' = 1, 2, \dots, J $, 
where 
\begin{align*}
\begin{aligned}
    M_{m}^{j,k} (t)&:= \int_0^t \left (\int_0^s \mathcal{D}_{m}^{j, k}(s, u) \sum_r \sigma_r^{k}(u) \,{\rm d}W_u^r \right) \sum_r \sigma_r^{j}(s)\,{\rm d}W_s^r,\\
     I_{m}^{1,j,k} (t)
        &=\int_0^t \left (\int_0^s \mathcal{D}_{m}^{j, k}(s,u) b^{j'} (u) \,{\rm d}u \right) \sum_r \sigma_r^{j}(s)\,{\rm d}W_s^r, \\
     I_{m}^{2,j,k} (t)
        &= \int_0^t \left (\int_0^s \mathcal{D}_{m}^{j, k}(s, u) \sum_r \sigma_r^{k}(u)\, {\rm d}W_u^r \right) b^j(s) \, {\rm d}s, 
        \end{aligned}
\end{align*}
and 
\begin{align*}
      I_{m}^{3,j,k} (t)
        &=\int_0^t \left (\int_0^s \mathcal{D}_{m}^{j, k}(s, u) b^{j'}(u)\, {\rm d}u \right)
        b^{j}(s) \,{\rm d}s.
\end{align*}
for $ j, k =1, 2, \cdots, J $,
and $ t \in [0,1] $. 
We assume the following  
\begin{ass}\label{H1}
{\rm (i)} For $p \geq 1 $, it holds that 
    \begin{align}\label{H1-1}
    A_p:= 
     {\bf E}\Big[(\sup_{t\in[0,1]} \sum_j|b^j(t)|^2)^{p/2}\Big] +           {\bf E}\Big[(\sup_{t\in[0,1]}
   \sum_{r,j} |\sigma_r^{j} (t)|^2)^{p/2}\Big]
    < +\infty. 
    \end{align}
{\rm (ii)} Each function $t \mapsto \sigma^j_r(t)$, $ j=1, 2, \dots, J, \, r=1, 2, \dots, d $, is continuous on $[0,1]$ almost surely. 
\end{ass}

\begin{lemma}\label{A1}
Under Assumption {\rm \ref{H1}}, we have, as $ m \to \infty $, 
\begin{align}\label{estI1I3}
   \mathbf{E} [| I_{m}^1 (t) |^2] + \mathbf{E} [| I_{m}^3 (t) |^2] 
   \leq O \left(\int_0^t  \left(\int_0^s |\mathcal{D}_{m}^{j, j'}(s, u) |\,{\rm d}u \right)^2 \mathrm{d}s \right), 
\end{align}
and
\begin{align}\label{estMI2}
   \mathbf{E} [| M_{m} (t) |^2] + \mathbf{E} [| I_{m}^2 (t) |^2] 
   \leq O \left(\int_0^t  \int_0^s |\mathcal{D}_{m}^{j, j'}(s, u) |^2\,{\rm d}u  \mathrm{d}s \right),
\end{align}
where $ O (\cdot) $ is Landau's big $ O $. Here we omit the superscript $ j, j'$ for 
clarity. 
\end{lemma}

\begin{proof}
    
By It\^o's isometry, we have 
\begin{align*}
    \begin{aligned}
         \mathbf{E} [| I_{m}^1 (t) |^2] 
         & \leq  \mathbf{E} \left[ \int_0^t \sum_r |\sigma^j_r (s) |^2 \left(\int_0^s  |\mathcal{D}_{m}^{j, j'}(s, u) ||b^{j'} (
         u)| \, \mathrm{d}u \right)^2 \mathrm{d}s \right] \\
         & \leq A_4 \left(\int_0^t  \int_0^s |\mathcal{D}_{m}^{j, j'}(s, u) |\,{\rm d}u \mathrm{d}s\right)^2 ,
    \end{aligned}
\end{align*}
and 
\begin{align*}
    \begin{aligned}
         \mathbf{E} [| I_{m}^3 (t) |^2] 
         &\leq \mathbf{E} \left[ \left(\int_0^t |b^{j} (s)|\int_0^s  |\mathcal{D}_{m}^{j, j'}(s, u) ||b^{j'} (
         u)| \, \mathrm{d}u \mathrm{d}s \right)^2 \right] \\
         & \leq A_4 \int_0^t  \left(\int_0^s |\mathcal{D}_{m}^{j, j'}(s, u) |\,{\rm d}u \right)^2 \mathrm{d}s, 
    \end{aligned}
\end{align*}
while with the Schwartz and  Burkh\"older--Davis--Gundy (BDG henceforth) inequality, 
we also obtain 
\begin{align*}
    \begin{aligned}
     & \mathbf{E} [| M_{m} (t) |^2] + \mathbf{E} [| I_{m}^2 (t) |^2] \\
   & \leq \int_0^1  \mathbf{E} \left[ \left(\int_0^s \mathcal{D}_{m}^{j, j'}(s, u)  \sum_r \sigma_r^{j'}(u) \,{\rm d}W_u^r \right)^4\right]^{1/2} \mathbf{E} \left[\left(\sum_r (\sigma_r^{j} (s))^2 + (b^{j} (s))^2 
        \right)^2 \right]^{1/2} \mathrm{d}s  
     \\
     &     \leq  C_{4,\mathrm{BDG}} \int_0^1  \mathbf{E} \left[ \left(\int_0^s (\mathcal{D}_{m_n}^{j, j'}(s, u))^2  \sum_r (\sigma_r^{j'}(u))^2 \,{\rm d}u \right)^2 \right]^{1/2}
     \sqrt{2} A_4^{1/2}  \mathrm{d}s \\ 
      & \leq C_{4,\mathrm{BDG}} \sqrt{2} A_4 \int_0^1  \int_0^s (\mathcal{D}_{m}^{j, j'}(s, u) )^2\,{\rm d}u \mathrm{d}s,
    \end{aligned}
\end{align*}
where $ C_{4,\mathrm{BDG}}$ is the universal constant 
appearing in the BDG inequality. 
\end{proof}

\subsection{Statement and a proof}

\begin{theorem}[Consistency of the estimator]\label{consistency with no micro}
Assume \eqref{rho}, \eqref{MC} 
and \eqref{sampligcond}.
Then, under Assumption {\rm \ref{H1}}, 
for $ j, j'=1, 2, \dots, J $, we have
\begin{align*}
         V_{n,m_n}^{j,j'} 
        &\to 
        \int_0^1 \Sigma^{j,j'} (s) \,\mathrm{d}s
\end{align*}
in probability as $ n \to \infty$.
\end{theorem}

\begin{proof}
The convergence to zero of the second term 
in \eqref{ItoF}
is 
seen by Lemma \ref{A1},
since 
$\mathcal{D}_{m_n}^{j,j'}(s,u) \to 0$ as $m_n \to \infty$ uniformly on every compact subset of $(0,1)^2$, 
and since $\mathcal{D}_{m_n}^{j,j'}(s,u)$ is bounded, the dominated convergence theorem implies 
$$
\int_0^1 \int_0^s (\mathcal{D}_{m_n}^{j,j'}(s,u))^2{\rm d}u \mathrm{d}s \to 0 
$$
as $ n \to \infty$.  
The convergence in $ L^1(P) $
of the first term in \eqref{ItoF}
to $ \int_0^1 \Sigma^{j,j'} (s) \mathrm{d}s $ 
is implied by 
\eqref{sampligcond}
and Assumption \ref{H1} (ii).
\end{proof}

\section{Asymptotic Normality}\label{sec:LTS}

\subsection{Discussions on the scale}

We start with a heuristic argument 
of finding the proper scale $ R_n $
such that 
\begin{align*}
        R_n \left( V_{n,m_n}^{j,j'}
        -\int_0^1 \mathcal{D}_{m_n}^{j, j'}(s, s) \Sigma^{j,j'} (s)  \, {\rm d}s\right),
    \end{align*}
converges stably in law to a
(conditioned) Gaussian variable. 
Looking at the decomposition \eqref{Fundecomp}, we see that 
the quadratic variation (process)
of $ M^{j,j'} $ 
\begin{align}\label{QVofM}
    \begin{aligned}
    =: \int_0^t \Big(\int_0^s 
        \mathcal{D}_{m_n}^{j, j'}(s,u)
         \mathcal{D}_{m_n}^{k, k'}(s,u)
         \Sigma^{j', k'} (u)
  \,{\rm d}u\Big)
         \Sigma^{j,k} (s) 
        \, {\rm d}s
         + \mathrm{Res}_t, \qquad t>0, 
    \end{aligned}
\end{align}
especially the first term, is the main term to control. 

The following is the first key to find the scale.
\begin{proposition}\label{PRSS}
Suppose that $ \rho_n m_n^2 \to 0 $, together with $ m_n \to \infty $ 
as $ n \to \infty $. Then, 
for any $ g \in C[0,1]^2 $, 
\begin{align*}
    \begin{aligned}
        m_n \int_0^1 \int_0^s 
        \mathcal{D}_{m_n}^{j, j'}(s,u)
         \mathcal{D}_{m_n}^{k, k'}(s,u)
        g(s,u)  \, {\rm d}s
  \,{\rm d}u
        \to 
        \int_0^1 g(s,s)
        \, {\rm d}s
    \end{aligned}
\end{align*}
as $ n \to \infty $.
\end{proposition}
A proof will be given in 
section \ref{PrPr} in the Appendices. 
The choice $ R_n = m_n^{1/2} $
is convincing once we establish the following. 
\begin{lemma}\label{lemma3.2}
    Suppose that 
    \begin{align*}
        \limsup_{n \to \infty} \rho_n m_n < \infty. 
    \end{align*}
Then, for $j, j'=1, 2, \dots, J$ and $p>1$, 
there exists a positive constant $C_p>0$,
only depending on the choice of $p>1$, such that
    \begin{equation*}
        \limsup_{m_n,n \to \infty}m_n \sup_{s \in [0,1]}
        \int_0^1\big|\mathcal{D}_{m_n}^{j, j'}(u, s)\big|^p\,{\rm d}u  \leq C_p.
    \end{equation*}
\end{lemma}
A proof will be given in section \ref{pf3.2}. The following is a direct consequence of Lemma \ref{lemma3.2}, given the estimate of \eqref{estI1I3}. 

\begin{cor}\label{cor3}
Under the same assumptions of 
Lemma {\rm \ref{lemma3.2}}, we have 
    \begin{align*}
m_n^{1/2} \mathbf{E} [|I^{1,j,j'} + I^{3,j,j'}|] \to 0\quad (n \to \infty)
\end{align*}
in probability. 
\end{cor}

\subsection{Discussions on the sampling scheme, continued}

The assumption in Proposition \ref{PRSS} is too demanding.
We again follow the strategy of \cite{ClGl} 
instead of proving. We assume the following.
\begin{ass}\label{SSfAN}
    There exist integrable functions $ \gamma^{j,j',k,k'} $
   on $ [0,1]$ such that
    \begin{align}\label{SS-AS}
        &  m_n \int_0^1 \int_0^s 
        \mathcal{D}_{m_n}^{j, j'}(s,u)
         \mathcal{D}_{m_n}^{k, k'}(s,u) \,{\rm d}u\,{\rm d}s
         \to \int_0^t \gamma^{j,j',k,k'} (s) \,{\rm d}s
     \end{align}
as $m_n, n \to \infty$. 
\end{ass}
The condition \eqref{SS-AS} is easier to check 
than the following. 
\begin{lemma}\label{lemma3.3+}
Assume \eqref{rho} and \eqref{SS-AS}. 
Then, for any 
 $ t \in [0, 1] $ and a continuous function $g:[0,1]^2 \to \mathbb{R}$, the following convergences as $ m, n \to \infty$ hold.
    \begin{equation}\label{convfan}
    \begin{aligned}
       &  m_n \int_0^t \int_0^s 
        \mathcal{D}_{m_n}^{j, j'}(s,u)
         \mathcal{D}_{m_n}^{k, k'}(s,u) g(s,u)\,{\rm d}u\,{\rm d}s
         \to \int_0^t \gamma^{j,j',k,k'} (s) g(s,s)\,{\rm d}s.
        \end{aligned}
        \end{equation}
\end{lemma}
\begin{proof}
For $j, j'=1, 2, \dots, J$,  $s \in [0,1)$ and $\varepsilon>0$, it holds that
\begin{equation}\label{keylm3+}
m_n \int_0^{s- \varepsilon} 
   | \mathcal{D}_{m_n}^{j, j'}(u, s)
     \mathcal{D}_{m_n}^{k, k'}(u, s)|\,{\rm d}u 
        \to 0 \quad  \text{as }m_n, n \to \infty.
\end{equation}
By the expression \eqref{new Dirichlet1}, 
we see that, for sufficiently large $ n $ and $u<s-\varepsilon$, it holds that
$ |\mathcal{D}_{m_n}^{j,j'} (u,s)| \leq C_\varepsilon m_n^{-1} $, where the constant $ C_\varepsilon $ only depends on $ \varepsilon $, from which the statement 
\eqref{keylm3+} immediately follows. 
That \eqref{keylm3+} implies \eqref{convfan} is also immediate. 
\end{proof}

Our strategy with Assumption \ref{SSfAN}
might be justified by a convincing example. 
Let us consider the case where 
\begin{align*}
   \Pi = \left\{ \left(\frac{k}{n}, \frac{k}{n}, \cdots, \frac{k}{n} \right) \in [0,1]^J:  k=0,1,\dots, n \right\},
\end{align*}
and for all $ j $, $ \varphi^j \equiv \varphi $ for some $ \varphi $,
that is, a synchronous sampling case. 
In this case, 
\begin{align*}
\begin{aligned}
   & \int_0^t \int_0^s  \mathcal{D}_{m_n}^{i, i'}(u, s)
        \mathcal{D}_{m_n}^{j, j'}(u, s)\, {\rm d}u\,{\rm d}s  \\
    &= \int_0^{[nt]/n} \left(\int_0^{[ns]/n} \left|\mathcal{D}_{m_n}\left( \varphi (u), \varphi(s) \right) \right|^2  {\rm d}u + \left(s- \frac{[ns]}{s}\right)\left|\mathcal{D}_{m_n}\left(\varphi (s), \varphi(s)\right) \right|^2\right)\,{\rm d}s  \\
    &+  \left(t- \frac{[nt]}{t}\right)\left(\int_0^{[nt]/n} \left|\mathcal{D}_{m_n} \left(\varphi (u), \varphi(t)\right) \right|^2  {\rm d}u\right)+ \left|\mathcal{D}_{m_n} \left(\varphi (t), \varphi(t)\right)\right|^2 \int_{[nt]/n}^t \left(s- \frac{[nt]}{n}\right)\,{\rm d}s.
\end{aligned} 
\end{align*}

\if1
\begin{ex}
We first consider the example of sampling studied in \cite{ClGl}. Let 
\begin{align*}
    \varphi \left( \left[\frac{k-1}{n},\frac{k}{n} \right) \right)
    = \frac{k-1}{n}, \quad k=1, \cdots, n,
\end{align*}
and 
\begin{align*}
     \rho_n m_n = \frac{m_n}{n} = a \in \mathbf{N}.
\end{align*}
Then, since
\begin{align*}
    \begin{aligned}
       & \mathcal{D}_{m_n} \left(\frac{k}{n}, \frac{k'}{n} \right)
        =  \mathcal{D}_{m_n} \left(\frac{ak}{m_n}, \frac{a k'}{m_n} \right)
        = \frac{1}{2m}\frac{\sin{2a\pi\left(k+k')\right)}}{\sin{\pi\left(k+k'\right)/(2n)}} + \frac{1}{2m}\frac{\sin{2a\pi\left(k-k'\right)}}{\sin{\pi\left(k-k'\right)/(2n)}} \\
    &= \begin{cases}
        0 & k\ne k', \\
        1& k = k'
    \end{cases},
    \end{aligned}
\end{align*}
we have
\begin{align*}
    \begin{aligned}
        & \int_0^t \int_0^s  \mathcal{D}_{m_n}^{i, i'}(u, s)
        \mathcal{D}_{m_n}^{j, j'}(u, s)\, {\rm d}u\,{\rm d}s  = 
             \frac{t}{2n}.
    \end{aligned}
\end{align*}
Thus, this example satisfy Assumption \ref{SSfAN} with $ \gamma (s) \equiv a/2 $.
\end{ex}
\fi
\begin{ex}
    Let us consider the case \eqref{KSSS};
\begin{align*}
    \varphi^j \left( \left[\frac{k-1}{n}, \frac{k}{n}\right) \right) = \frac{2k-1}{2n+1},
\end{align*}
and 
\begin{align*}
    \frac{2m}{2n+1} = a \in \mathbf{N}. 
\end{align*}
In this case, since
\begin{align*}
    \begin{aligned}
       & \mathcal{D}_{m_n} \left(\frac{2k-1}{2n+1}, \frac{2k'-1}{2n+1} \right)
        =  \mathcal{D}_{m_n} \left(\frac{a(2k-1)}{2m_n}, \frac{a (2k'-1)}{2m_n} \right) \\
       & = \frac{1}{2m}\frac{\sin{2a\pi\left(k+k'-1)\right)}}{\sin{\pi(k+k'-1)/(2n+1)}} + \frac{1}{2m}\frac{\sin{2a\pi\left(k-k'\right)}}{\sin{\pi\left(k-k'\right)/(2n+1)}} 
    = \begin{cases}
        0 & k\ne k' \\
        1 & k=k', 
    \end{cases}
    \end{aligned}
\end{align*}
we have 
\begin{align*}
    \begin{aligned}
        & \int_0^t \int_0^s  \mathcal{D}_{m_n}^{i, i'}(u, s)
        \mathcal{D}_{m_n}^{j, j'}(u, s)\, {\rm d}u\,{\rm d}s  =  \frac{t}{2n}.
    \end{aligned}
\end{align*}
Thus, it satisfies Assumption {\rm \ref{SSfAN}}
with $ \gamma (s) \equiv a/2 $.
\end{ex}

\subsection{More on the estimates on the residues; we may need a bit of Malliavin calculus}

Contrary the case of $ I^1 $ and $ I^3 $, 
the combination of 
the estimate \eqref{estMI2} and Lemma \ref{lemma3.2} 
is insufficient to prove the convergence $ m_n^{1/2} \mathbf{E} [|I^2|] \to 0 $. 
Instead of the standard ``BDG approach" taken in the proof of Lemma \ref{A1}, we resort to a bit of Malliavin calculus, its integration by parts (IBP for short) formula to be precise,
which is the approach taken in \cite{ClGl}\footnote{To be precise, they used the IBP technique to prove the convergence of $ \mathrm{Res}_t $ in \eqref{QVofM} and $ \langle M, W\rangle $, but used another approach to prove the convergence of $ I^2 $ for which the Malliavin differentiability for $ b $ is not required.}. 
Specifically, to estimate $ \mathbf{E}[|I^2_{m_n}|^2] $, we use the IBP instead of Schwartz inequality to get 
\begin{align}\label{IBP-1}
    \begin{aligned}
        (\mathbf{E}[|I^{2,j,k}_{m_n}|^2]
      &=) \int_{[0,1]^2}
        \mathbf{E} \left[\left (\int_0^s \mathcal{D}_{m_n}^{j, k}(s, u) \sum_r \sigma_r^{k}(u)\, {\rm d}W_u^r \right) 
        L_{m_n}^{j,k}(s',s')
        b^j(s) b^j(s') \right]\mathrm{d}s 
        \mathrm{d}s' \\
&= \int_{[0,1]^2}
        \mathbf{E} \left[ \int_0^s \mathcal{D}_{m_n}^{j, k}(s, u) \sum_r \sigma_r^{k}(u)\, 
      \nabla_{u,r} (  L_{m_n}^{j,k}(s',s')
        b^j(s) b^j(s') ) \mathrm{d}u \right]\mathrm{d}s 
        \mathrm{d}s',
    \end{aligned}
\end{align}
where 
\begin{align}
        L_{m_n}^{j,j'}(t,s):=\int_0^s \mathcal{D}_{m_n}^{j, j'}(t, u)
       \sum_r \sigma_r^{j'}(u) \,{\rm d}W_u^r, \qquad s,t \in [0, 1],
\end{align}
and $ \nabla_{u,r} $ denotes\footnote{Here we avoid using the commonly used notation $ D $ for the derivative so as not to mix it up with the Dirichlet kernel.}
the Malliavin--Shigekawa derivative 
in the direction of ``$ \mathrm{d} W_u^r $".
The merit of the expression in the right-hand-side of \eqref{IBP-1} is that we obtain an estimate with $ \int |\mathcal{D}| $
instead of $ \int |\mathcal{D}|^2 $, 
though we need to assume further some differentiability and integrability of $ b $ and $ \sigma $\footnote{As is remarked in \cite{ClGl}, the  conditions \eqref{H3} and \eqref{H3b} are not too strong; for example 
the solution for a stochastic differential equation with  smooth bounded coefficients naturally satisfy these conditions.};
\begin{ass}\label{MalDer}
{\rm (i)} For any $p>1$, 
$\sigma^j_r \in \mathbb{D}^{1,p}$, $j=1, 2, \dots, J$, 
$r=1, 2, \dots, d$, 
and it holds that
    \begin{equation}\label{H3}
         {\bf E}\Big[
         (\sup_{s,t\in[0,1]}\sum_{r,r',j}| \nabla_{s,r'} \sigma_r^j(t)|^2)^{p/2} \Big]  < +\infty,
    \end{equation}
where $\mathbb{D}^{1, p}$ stands for the domain of the Malliavin derivative $ \nabla $ in $L^p(\Omega)$ 
{\rm (}see {\rm \cite{ND}} for details{\rm )}.

\vspace{1mm}
\noindent
{\rm (ii)} For any $p>1$, 
$b^j \in \mathbb{D}^{1,p}$, and $j=1, 2, \dots, J$, 
it holds that
    \begin{equation}\label{H3b}
         {\bf E}\Big[
    (\sup_{s,t\in[0,1]}\sum_{r,j}| \nabla_{s,r} b^j(t)|^2)^{p/2} \Big]  < +\infty.
    \end{equation}
\end{ass}
\begin{lemma}\label{cor4}
Under Assumptions {\rm \ref{H1}} and {\rm \ref{MalDer}}, we have
\begin{align*}
    m_n \mathbf{E} [| I_{m_n}^2 (t) |^2] \to 0 \quad (m_n,n \to \infty). 
\end{align*}
\end{lemma}

\begin{proof}
We start with \eqref{IBP-1}. 
We can go further;
\begin{align*}
    \begin{aligned}
      \mathbf{E}[|I^{2,j,k}_{m_n}|^2] = \int_{[0,1]^2}
 \int_0^s \mathcal{D}_{m_n}^{j, k}(s, u) 
\mathbf{E}\left[ \sum_r
\sigma_r^{k}(u)\Phi_r(u)b^j(s) b^j(s') \right] 
\mathrm{d}u \mathrm{d}s 
        \mathrm{d}s',
    \end{aligned}
\end{align*}
where 
\begin{align*}
    \begin{aligned}
        & \Phi_r(u) \\
        &:= 
1_{\{u\leq s'\}} \left( \mathcal{D}_{m_n}^{j, k}(s', u) \sigma^k_r(u) 
+ \int_0^{s'}\mathcal{D}_{m_n}^{j, k}(s', u') \sum_{r'} \nabla_{u,r} \sigma^k_{r'} (u') \mathrm{d}W^r_{u'} \right)+ L_{m_n}^{j,k}(s',s')\nabla_{u,r}.
    \end{aligned}
\end{align*}
Since it holds that for 
\begin{align*}
    \begin{aligned}
        & (\mathbf{E}[|L_{m_n}^{j,k}(s',s')|^p])^{2/p} +
        \left(\mathbf{E}[| \int_0^{s'}\mathcal{D}_{m_n}^{j, k}(s', u') \sum_{r'} \nabla_{u,r} \sigma^k_{r'} (u') \mathrm{d}W^r_{u'}|^p] \right)^{2/p} \\ 
      & \hspace{2cm} = O \left( \int_0^s |\mathcal{D}_{m_n}^{j, k}(s, u)|^2 \mathrm{d}u
        \right), \qquad p>1,
    \end{aligned}
\end{align*}
by BDG and Assumptions \ref{H1} and \ref{MalDer}, 
we have 
\begin{align*}
    \begin{aligned}
    \mathbf{E}[|I^{2,j,k}_{m_n}|^2]
    &= O\bigg( \int_{[0,1]^2}\int_0^{s\wedge s'}
    |\mathcal{D}_{m_n}^{j, k}(s, u) |
    |\mathcal{D}_{m_n}^{j, k}(s', u)|
    \mathrm{d}u \mathrm{d}s 
        \mathrm{d}s'\\
        &+ \int_{[0,1]^2}\int_0^{s}
    |\mathcal{D}_{m_n}^{j, k}(s, u) |
    \mathrm{d}u 
   \left( \int_0^{s'} |\mathcal{D}_{m_n}^{j, k}(s', u')|^2
    \mathrm{d}u'\right)^{1/2} \mathrm{d}s\bigg),
    \end{aligned}
\end{align*}
which is seen to be $ o (m_n) $
by Lemma \ref{lemma3.2}.
\end{proof}

\subsection{Statement and a proof}

The error distribution is 
obtained as
\begin{theorem}[Asymptotic normality]\label{CLT with no micro}
Under Assumptions {\rm \ref{H1}}, {\rm \ref{SSfAN}}, and {\rm \ref{MalDer}}, for $j, j'=1, 2, \dots, J$,
the sequence of random variables
    \begin{align*}
        m_n^{1/2}\left( V_{n,m_n}^{j,j'}
        -\int_0^1 \mathcal{D}_{m_n}^{j, j'}(s, s) \Sigma^{j,j'} (s)  \, {\rm d}s\right),
    \end{align*}
converges to 
    $$
    \begin{aligned}
        \int_0^1 
        \sqrt{(\gamma^{j,j',j,j'} (s) + \gamma^{j',j,j',j}
        \Sigma^{j,j}(s)  \Sigma^{j',j'}(s) 
       + 2 \gamma^{j,j',j',j}(s) (\Sigma^{j,j'} (s))^2 }
\, {\rm d}B_s
    \end{aligned}
    $$
    stably in law as $m_n,n \to \infty$,
    where $B=(B_t)_{t \in [0.\, 1]}$ is a one-dimensional Brownian motion independent of $(W_t)_{t \in [0, 1]}$.
\end{theorem}

\begin{proof}

Given Corollary \ref{cor3}, Lemma \ref{cor4}, and Assumption \ref{SSfAN}, 
it suffices to show that 
$ m_n \mathbf{E}[ |\mathrm{Res}_t |] \to 0 $
and 
    \begin{align*}
        &{\bf E}\Big[\langle m_n^{1/2} M^{j,j'}_{m_n},W^r \rangle^2_t\Big]\nonumber\\
&=m_n \int_{[0,t]^2} {\bf E}
        \Big[L_{m_n}^{j,j'}(s,s)L_{m_n}^{j,j'}(s',s')\sigma_r^{j}(s)\sigma_r^{j}(s')\Big]
        {\rm d}s \,{\rm d}s' \to 0 \label{new notaion}
    \end{align*}
as $m_n, n \to \infty $ (by Jacod's theorem \cite{JJ},
see also \cite{JP}), 
but totally the same 
proof as the one in \cite{ClGl}
works, and so we omit it. 
\end{proof}


\color{black}

\appendix

\section{Appendices}

\subsection{A proof of Lemma \ref{lemma3.2}}\label{pf3.2}
Let 
\begin{align*}
    D_m (x) := \frac{1}{2m} \frac{\sin m \pi x}{\sin \frac{\pi}{2}x}.
\end{align*}
By extending $ \varphi^j $ and $ \varphi^{j'} $ periodically, $ \mathcal{D}^{j,j'}_m (u,s) = D_m(\varphi^j(u) + \varphi^{j'} (s)) + D_m(\varphi^j(u) - \varphi^{j'}(s)) $ is periodic in both $ u $ and $ s $ with the period $ 2 $, 
and therefore 
\begin{align*}
& \sup_{s \in [0,1]}
        \int_0^1\big|\mathcal{D}_{m_n}^{j, j'}(u, s)\big|^p\,{\rm d}u \\
        & \leq \sup_{c \in \mathbf{R}}
       \left( \int_{c-1}^{c+1} \big|D_{m_n}(\varphi^j(u)-c)\big|^p\,{\rm d}u + 
        \int_{-c-1}^{-c+1}\big|D_{m_n}(\varphi^j(u)+c)\big|^p\,{\rm d}u \right).
\end{align*}
Therefore, it is sufficient to show that
    \begin{equation}\label{bound1}
        \limsup_{m_n,n \to \infty}m_{n}\sup_{c \in \mathbb{R}}
        \int_{c-1}^{c+1}
        \big|D_{m_n}(\varphi^j(u)-c)\big|^p {\rm d}u < \infty.
    \end{equation}
Put 
$a:=\limsup_{m_n, n \to \infty} m_n \rho_n $ which is in $ [ 0, \infty) $ by the assumption, 
and let
    \begin{align*}
        J_n^{1}(u):=
        \int_{c-\frac{1}{2}}^{c+\frac{1}{2}}
        \bm{1}_{\{|u-c|>\frac{2a+2}{m_n}\}}\big|D_{m_n} ( \varphi^j(u)-c)\big|^p {\rm d}u
    \end{align*}
and 
\begin{align*}
       J_n^{2}(u):= \int_{c-1}^{c+1}
        \bm{1}_{
        \{|u-c|\leq\frac{2a+2}{m_n}\}}\big|D_{m_n} ( \varphi^j(u)-c)\big|^p {\rm d}u.
\end{align*}
For $m_n$ and $n$ large enough, we see that  
    \begin{align}\label{mn2}
    \begin{aligned}
            |u-c| &=\big|\varphi^{j}(u)-c-(\varphi^{j}(u)-u)| 
        \leq\big|\varphi^{j}(u)-c\big|+
     \big|\varphi^{j}(u)-u\big|\\
     &\leq\big|\varphi^{j}(u)-c\big|+\rho_n.
         \end{aligned}
    \end{align}
Therefore we obtain
\begin{align}\label{2pa}
\begin{aligned}
    |u-c|>\frac{2a+2}{m_n} &\Rightarrow \big|\varphi^{j}(u)-c\big|
\geq\frac{2}{m_n}+\frac{2a-m_n\rho_n}{m_n}>\frac{2+a}{m_n} \\
& \Rightarrow \frac{1}{2 m_n  \big|\varphi^{j}(u)-c\big|} \leq \frac{1}{2} \frac{1}{2+a}  < 1.
\end{aligned}
\end{align}
Since  
    \begin{equation}\label{Dirbound}
        \big|D_{m_n}(x)\big| \leq 1 \wedge \frac{1}{2m_n x} 
    \end{equation}
for $ x \in \mathbf{R} $, it follows from \eqref{2pa} that
\begin{equation}\label{bound3}
    \begin{split}
        J_n^{1}(u) \leq
        \int_{c-1}^{c+1}
        \bm{1}_{\{|u-c|>\frac{2a+2}{m_n}\}}\left|\frac{1}{2m_n(\varphi^j(u)-c)}\right|^p {\rm d}u.
        \end{split}
\end{equation}

Since \eqref{mn2} implies, for sufficiently large $ n $, 
$|u-c|>\frac{2a+2}{m_n} \Rightarrow\left|\varphi^{j}(u)-c\right|\geq|u-c|-\frac{2a}{m_n}$, one has
    \begin{align*}   
    \begin{aligned}
&\int_{c-1}^{c+1}
\bm{1}_{\{|u-c|>\frac{2a+2}{m_n}\}}\left|\frac{1}{2m_n(\varphi^j(u)-c)}\right|^p {\rm d}u \\
        &\leq\int_{c-1}^{c+1}
        \bm{1}_{\{|u-c|>\frac{2a+2}{m_n}\}}\left|\frac{1}{2m_n\left(|u-c|-\frac{2a}{m_n}\right)}\right|^p {\rm d}u,\\
        &= \int_{c-1}^{c+1}
     \bm{1}_{\{\frac{m_n}{2}|u-c|-a>1\}}\frac{1}{\left|4\left(\frac{m_n}{2}|u-c|-a\right)\right|^p} \,{\rm d}u, \\
\end{aligned}
\end{align*}
(by changing variables with $ w = \frac{m_n}{2} |u-c|-a$) 
\begin{align} \label{bound4}
\begin{aligned}
        &= 4 \int_{-a}^{\frac{m_n}{2}-a} 1_{\{w > 1\}}
   \frac{1}{4^p m_n} w^{-p} \,{\rm d}w \leq\left(\frac{1}{4}\right)^p\frac{4}{m_n}
        \int_1^{\infty}\omega^{-p}{\rm d}\omega.
        \end{aligned}
        \end{align}
This establishes \eqref{bound1}
since clearly one has, by \eqref{Dirbound}, 
\begin{align*}
    \begin{aligned}
    J_n^2(u) \leq   \int_{c-1}^{c+1}
        \bm{1}_{\{|u-c|\leq\frac{2a+2}{m_n}\}} {\rm d}u
        \leq \frac{4a+4}{m_n}.
    \end{aligned}
\end{align*}
\qed

\if2
\subsection{Proof of Lemma \ref{A1}}

Let us recall 
that in the proof of Theorem \ref{CLT with no micro} we have put
\begin{align*}
\begin{aligned}
    M_{m_n}(t)&= \int_0^t \left (\int_0^s \mathcal{D}_{m_n}^{j, j'}(s, u) \sum_r \sigma_r^{j'}(u) \,{\rm d}W_u^r \right) \sum_r \sigma_r^{j}(s)\,{\rm d}W_s^r,\\
     I_{m_n}^1(t)
        &=\int_0^t \left (\int_0^s \mathcal{D}_{m_n}^{j, j'}(s,u) b^{j'}(u) \,{\rm d}u \right) \sum_r \sigma_r^{j}(s)\,{\rm d}W_s^r, \\
     I_{m_n}^2(t)
        &= \int_0^t \left (\int_0^s \mathcal{D}_{m_n}^{j, j'}(s, u) \sum_r \sigma_r^{j'}(u)\, {\rm d}W_u^r \right) b^j(s) \, {\rm d}s, 
        \end{aligned}
\end{align*}
and 
\begin{align*}
      I_{m_n}^3(t)
        &=\int_0^t \left (\int_0^s \mathcal{D}_{m_n}^{j, j'}(s, u) b^{j'}(u)\, {\rm d}u \right)
        b^{j}(s) \,{\rm d}s.
\end{align*}
Here we omit the superscript $ j, j'$ for 
clarity. 

By It\^o's isometry, we have 
\begin{align*}
    \begin{aligned}
         \mathbf{E} [| I_{m_n}^1 (t) |^2] 
         & \leq  \mathbf{E} \left[ \int_0^t \sum_r |\sigma^j_r (s) |^2 \left(\int_0^s  |\mathcal{D}_{m_n}^{j, j'}(s, u) ||b^{j'} (
         u)| \, \mathrm{d}u \right)^2 \mathrm{d}s \right] \\
         & \leq A_4 \left(\int_0^t  \int_0^s |\mathcal{D}_{m_n}^{j, j'}(s, u) |\,{\rm d}u \mathrm{d}s\right)^2 ,
    \end{aligned}
\end{align*}
and 
\begin{align*}
    \begin{aligned}
         \mathbf{E} [| I_{m_n}^3 (t) |^2] 
         &\leq \mathbf{E} \left[ \left(\int_0^t |b^{j} (s)|\int_0^s  |\mathcal{D}_{m_n}^{j, j'}(s, u) ||b^{j'} (
         u)| \, \mathrm{d}u \mathrm{d}s \right)^2 \right] \\
         & \leq A_4 \int_0^t  \left(\int_0^s |\mathcal{D}_{m_n}^{j, j'}(s, u) |\,{\rm d}u \right)^2 \mathrm{d}s, 
    \end{aligned}
\end{align*}
while with the Schwartz and  Burkh\"older--Davis--Gundy inequality, 
we also obtain 
\begin{align*}
    \begin{aligned}
     & \mathbf{E} [| M_{m_n} (t) |^2] + \mathbf{E} [| I_{m_n}^2 (t) |^2] \\
   & \leq \int_0^1  \mathbf{E} \left[ \left(\int_0^s \mathcal{D}_{m_n}^{j, j'}(s, u)  \sum_r \sigma_r^{j'}(u) \,{\rm d}W_u^r \right)^4\right]^{1/2} \mathbf{E} \left[\left(\sum_r (\sigma_r^{j} (s))^2 + (b^{j} (s))^2 
        \right)^2 \right]^{1/2} \mathrm{d}s  
     \\
     &     \leq  C_{4,\mathrm{BDG}} \int_0^1  \mathbf{E} \left[ \left(\int_0^s (\mathcal{D}_{m_n}^{j, j'}(s, u))^2  \sum_r (\sigma_r^{j'}(u))^2 \,{\rm d}u \right)^2 \right]^{1/2}
     \sqrt{2} A_4^{1/2}  \mathrm{d}s \\ 
     \\ & \leq C_{4,\mathrm{BDG}} \sqrt{2} A_4 \int_0^1  \int_0^s (\mathcal{D}_{m_n}^{j, j'}(s, u) )^2\,{\rm d}u \mathrm{d}s,
    \end{aligned}
\end{align*}
where $ C_{4,\mathrm{BDG}}$ is the universal constant 
appearing in the Burkh\"older--Davis--Gundy inequality. 
\qed 
\fi

\subsection{A proof of Proposition \ref{PRSS}}\label{PrPr}

Under the condition that  $ \rho_n m_n^2 \to 0 $, we have, by a similar argument as the one we did for the proof of Lemma \ref{l21},
\begin{align*}
   m_n  \int_0^1 \int_0^u  \left(\mathcal{D}^{j,j'} (s,u) \mathcal{D}^{k,k'} (s,u)  -(\mathcal{D} (s,u))^2 \right) f(s,u) \,\mathrm{d}u \mathrm{d}s \to 0 \quad (n \to \infty).
\end{align*}
Therefore, it suffices to prove 
\begin{align*}
&  m \int_0^1 \int_0^u (\mathcal{D} (s,u))^2 f(s,u) \,\mathrm{d}u \mathrm{d}s -\frac{1}{2}\int_0^1 f(s,s) \, \mathrm{d}s \to 0 \quad (n \to \infty). 
\end{align*}
We note that, 
extending $ f(s,u) $ from $ \{ (s,u) : s\leq u \} $ to $ [0,1] $ symmetrically, 
\begin{align*}
    \begin{aligned}
       \int_0^1 \int_0^u (\mathcal{D} (s,u))^2 f(s,u) \,\mathrm{d}u \mathrm{d}s 
        = \frac{1}{2} \int_0^1 \int_0^1 |\mathcal{D} (s,u)|^2 f(s,u) \,\mathrm{d}u \mathrm{d}s.
    \end{aligned}
\end{align*}
Then, 
letting
\begin{align*}
    g(s) :=m\int_0^1 |\mathcal{D} (u,s)|^2\,\mathrm{d}u,
\end{align*}
we have 
\begin{align}
&  m \int_0^1 \int_0^u (\mathcal{D} (s,u))^2 f(s,u) \,\mathrm{d}u \mathrm{d}s -\frac{1}{2}\int_0^1 f(s,s) \, \mathrm{d}s  \nonumber \\
&= \frac{m}{2} \int_0^1 \int_0^1 |\mathcal{D} (s,u)|^2 ((f(s,u) -f(s,s))\,\mathrm{d}u \mathrm{d}s  -\frac{1}{2}\int_0^1 (1-g(s))f(s,s) \, \mathrm{d}s. \label{1-g}
\end{align}
By the expression \eqref{D-exp1},
\begin{align*}
    \begin{aligned}
        & 
        g(u) \\
       & = \frac{4}{m}\sum_{l=1}^{m}
        \sum_{l'=1}^{m}\cos\left(l-\frac{1}{2}\right)\pi u \cos\left(l'-\frac{1}{2}\right)\pi u
        \int_0^1 \cos\left(l-\frac{1}{2}\right)\pi s \cos\left(l'-\frac{1}{2}\right)\pi s \,\mathrm{d}s.
    \end{aligned}
\end{align*}
Since it holds that
\begin{align*}
    \begin{aligned}
       & \int_0^1 \cos\left(l-\frac{1}{2}\right)\pi s \cos\left(l'-\frac{1}{2}\right)\pi s \,\mathrm{d}s \\
       &= \frac{1}{2} \int_0^1 \left(\cos(l+l'-1)\pi s + \cos(l-l')\pi s \right)\,\mathrm{d}s \\
       & = \begin{cases}
           1/2 & l=l' \\
           0 & l \ne l'
       \end{cases},
    \end{aligned}
\end{align*}
we have
\begin{align*}
    \begin{aligned}
        & g(u) 
        = \frac{2}{m}\sum_{l=1}^{m}\cos^2\left(l-\frac{1}{2}\right)\pi u 
        \\
        &= 1 + \frac{1}{m} \sum_{l=1}^m 
        \cos (2l-1) \pi u \\
        &= 1+ \frac{1}{2m} \frac{\sin 2 m \pi u}{\sin 2 \pi u} 
        = 1 + D_m (2u). 
    \end{aligned}
\end{align*}
Then, by Lemma \ref{lemma3.2}, 
we see that 
\begin{align*}
    \frac{1}{2}\int_0^1 (1-g(s))f(s,s) \, \mathrm{d}s \to 0 \quad (m \to \infty).
\end{align*}

Finally we shall prove 
the convergence of the first term in \eqref{1-g}. Recalling \eqref{key equality}, we have
\begin{align}\label{2a2}
    \begin{aligned}
      & m  \left| \int_0^1 \int_0^1 |\mathcal{D} (s,u)|^2 (f(s,u) -f(s,s))\,\mathrm{d}u \mathrm{d}s \right| \\
     &\leq \int_{[0,1]^2} 2m ((D_m (u+s))^2 
         + (D_m (u-s))^2) | f(s,u) -f(s,s) |
         \,\mathrm{d}u \mathrm{d}s.
    \end{aligned}
\end{align}
We rely on the uniformly continuity of $ f $. 
For arbitrary sufficiently small $ \varepsilon > 0 $, 
we can take $ \delta > 0 $
such that 
\begin{align*}
 |s-u| <\delta 
  \Rightarrow |f(s,u)-f(s,s)| < \varepsilon.
\end{align*}
Let 
\begin{align*}
    A^+_\delta := \{ (s,u) \in [0,1]^2:
    \delta/2 < s+u < \delta/2 \},
\end{align*}
and 
\begin{align*}
    A^-_\delta := \{ (s,u) \in [0,1]^2:
    |s-u| < \delta \}.
\end{align*}
Then clearly $ (s,u) \in A^{\pm}_\delta $ 
satisfies $ |s-u|< \delta $ and therefore 
$  |f(s,u)-f(s,s)| < \varepsilon $.
Then, we can bound the right-hand-side of \eqref{2a2} by
\begin{align*}
    \begin{aligned}
    4 \Vert f\Vert_\infty \left(\int_{A^+_\delta} \frac{1}{2m} 
    \frac{\sin^2 m \pi (u+s)}{\sin^2 \pi (u+s)/2}\,\mathrm{d}u \mathrm{d}s
    + 
     \int_{A^-_\delta} \frac{1}{2m} 
   \frac{\sin^2 m \pi (u-s)}{\sin^2 \pi (u-s)/2}\,\mathrm{d}u \mathrm{d}s \right) \\
    + \varepsilon \left(\int_{[0,1]^2} \frac{1}{2m} \left(
    \frac{\sin^2 m \pi (u+s)}{\sin^2 \pi (u+s)/2} +
    \frac{\sin^2 m \pi (u-s)}{\sin^2 \pi (u-s)/2} \right)\,\mathrm{d}u \mathrm{d}s \right).    
    \end{aligned}
\end{align*}
Since 
\begin{align*}
    \begin{aligned}
        \int_{A^+_\delta} \frac{1}{2m} 
    \frac{\sin^2 m \pi (u+s)}{\sin^2 \pi (u+s)/2}\,\mathrm{d}u \mathrm{d}s
    +  \int_{A^-_\delta} \frac{1}{2m} 
   \frac{\sin^2 m \pi (u-s)}{\sin^2 \pi (u-s)/2}\,\mathrm{d}u \mathrm{d}s
    \leq 2 \int_\delta^{1} \frac{1}{2my^2}\,\mathrm{d}y \leq \frac{1}{m\delta}
    \end{aligned}
\end{align*}
and
\begin{align*}
    \begin{aligned}
      &  \int_{[0,1]^2} \frac{1}{2m} \left(
    \frac{\sin^2 m \pi (u+s)}{\sin^2 \pi (u+s)/2} +
    \frac{\sin^2 m \pi (u-s)}{\sin^2 \pi (u-s)/2} \right)\,\mathrm{d}u \mathrm{d}s \\
    &= \frac{1}{2m} \int_{[0,1]^2} \left(\left(\sum_{l=-m+1}^m e^{(l-\frac{1}{2})\pi (s+u)} 
    \right)^2 + \left(\sum_{l=-m+1}^m e^{(l-\frac{1}{2})\pi (s-u)} \right)^2\right)
    \,\mathrm{d}u \mathrm{d}s 
    = \frac{1}{2},
    \end{aligned}
\end{align*}
we have
\begin{align*}
    \begin{aligned}
         & m  \left| \int_0^1 \int_0^1 |\mathcal{D} (s,u)|^2 (f(s,u) -f(s,s))\,\mathrm{d}u \mathrm{d}s \right| \leq \frac{4 \Vert f\Vert_\infty}{m\delta} + \frac{\varepsilon}{2},
    \end{aligned}
\end{align*}
which shows the convergence to zero (as $ m \to \infty $) of the first term in \eqref{1-g}.
\qed
\end{document}